\documentclass[11pt]{amsart}
\usepackage{amsthm,amsfonts,amsmath,amssymb}
\usepackage[utf8]{inputenc}
\usepackage{bm}
\usepackage{mathtools}%floor
\usepackage{enumitem}
\usepackage{listings}
\usepackage{mathrsfs}
\usepackage{fancybox}
\usepackage{dsfont}
\usepackage{float}
\usepackage{arcs}
\usepackage{pgf,tikz}
\usetikzlibrary{arrows}
\usepackage{graphicx}
\usepackage{color}
\usepackage[all]{xy}

\usepackage{tikz}
   \usetikzlibrary{calc}

\newcommand{\C}{\mathbb{C}}
\newcommand{\R}{\mathbb{R}}
\newcommand{\Z}{\mathbb{Z}}
\newcommand{\cS}{\mathcal{S}}
\newcommand{\cH}{\mathcal{L}}
\newcommand{\cC}{\mathcal C}
\newcommand{\cD}{\mathcal D}
\newcommand{\ct}{C}
\newcommand{\st}{S}

\newcommand{\N}{\mathbb{N}}

\newcommand{\cp}[1]{\C P^{#1}}

\newcommand{\jm}{j_{\s{-}}}
\newcommand{\jp}{j_{\s{\s{+}}}}

\renewcommand{\l}[1]{{\color{blue} \sf  [#1]}}

\newcommand{\s}[1]{\scalebox{0.7}{#1}}

\definecolor{bluegray}{rgb}{0.4, 0.6, 0.8}
\renewcommand{\emph}[1]{{\it {\color{bluegray} #1}}}

\newcommand{\V}{\mathscr V}

\DeclareMathOperator{\Log}{Log}
\DeclareMathOperator{\Arg}{Arg}
\DeclareMathOperator{\im}{im}
\DeclareMathOperator{\ind}{ind}
\DeclareMathOperator{\lcm}{lcm}
\DeclareMathOperator{\id}{id}

\DeclareMathOperator{\Sym}{Sym}

\newtheorem{Proposition}{Proposition}[section]

\newtheorem{Definition}[Proposition]{Definition}
\newtheorem{Lemma}[Proposition]{Lemma}
\newtheorem{Remark}[Proposition]{Remark}
\newtheorem{Theorem}{Theorem}

\newtheorem{Corollary}[Theorem]{Corollary}

\date{}

\begin{document}
 \title{Braid monodromy of univariate fewnomials}

\author{Alexander Esterov and  Lionel Lang}

\maketitle

\begin{abstract}
Let $\cC_d\subset \C^{d+1}$ be the space of non-singular, univariate polynomials of degree $d$. The Vi\`{e}te map $\V : \cC_d \rightarrow \Sym_d(\C)$ sends a polynomial to its unordered set of roots. It is a classical fact that the induced map $\V_*$  at the level of fundamental groups realises an isomorphism between $\pi_1(\cC_d)$ and the Artin braid group $B_d$. For fewnomials, or equivalently for the intersection $\cC$ of $\cC_d$ with a collection of coordinate hyperplanes in $\C^{d+1}$, the image of the map  $ \V_* : \pi_1(\cC) \rightarrow B_d$ is not known in general.

In the present paper, we show that the map $\V_*$ is surjective provided that the support of the corresponding polynomials spans $\Z$ as an affine lattice. If the support spans a strict sublattice of index $b$, we show that the image of $\V_*$ is the expected wreath product of $\Z/b\Z$ with $B_{d/b}$. From these results, we derive an application to the computation of the braid monodromy  for collections of univariate polynomials depending on a common set of parameters.
\end{abstract}

\footnote{{\bf MSC-2010 Classification:} 
20F36, %Braid groups; Artin groups
55R80, %Discriminantal varieties, configuration spaces
14T05 %Tropical geometry
}

\section{Introduction}

\subsection{Braid monodromy of univariate polynomials}
Let $\cC_d\subset \C^{d+1}$ denote the space of non-singular, \l{monic} univariate polynomials of degree $d$. The Vi\`{e}te map $\V$ that associates to a  polynomial $p(x)\in \cC_d$ its unordered set of roots realises an isomorphism between $\cC_d$ and the \emph{configuration space} $C(\C,d)$ of $d$ distinct points in $\C$. In turn, the induced map $ \V_* : \pi_1(\cC_d) \rightarrow \pi_1\big(C(\C,d)\big)$ is an isomorphism onto the \emph{Artin braid group} $B_d:=\pi_1\big(C(\C,d)\big)$. Moreover, the higher homotopy groups of $\cC_d$ vanish so that $\cC_d$ is the Eilenberg-MacLane space $K(B_d,1)$, see \cite{FN}. For instance, this fact was succesfully used by V.I. Arnol'd in \cite{Arn71} to compute the cohomology groups of $B_d$.

From the perspective of fewnomial theory, it is natural to consider the related problem of determining the fundamental group (as well as higher homotopy groups) of the space $\cC_A\subset \C^{d+1}$ consisting of polynomials $p(x)=\sum_{a\in A} c_ax^a$ with a given set of exponents $A\subset \{0,\cdots,d\}$.
In particular, this study fits into the general program of determining the fundamental group of the complement to discriminant varieties, see \cite{DL}. The particular instance of the space $\cC_A$ was considered in \cite{Lib90} in relation to Smale-Vassiliev's complexity for algorithms. In the latter, A. Libgober asked for the computation of $\pi_1(\cC_A)$ and whether $\cC_A$ is a $K(G,1)$-space. He answered both questions in the case of trinomials, i.e. $\#A=3$. For general supports $A$, it seems that not much is known about $\pi_1(\cC_A)$. 

As a first approximation, we suggest to investigate the image of $\pi_1(\cC_A)$ under the map $\V_*$. This study generalises the determination of the Galois group of the universal polynomial with support $A$ (see for instance \cite{Co}) whose multivariate analogues were studied in \cite{E18} and \cite{AL}.

\subsection{Main results}
Since multiplying polynomials with monomials does not affect the set of roots in $\C^\star$,  there is no loss of generality in restricting to supports $A$ such that $\{0,d\}\subset A \subset \{0,\cdots,d\}$. We  define  the \emph{space of conditions}  $\cC_A \subset \C^A$ to be the space of polynomial $p(x)$ with $d$ distinct roots in $\C^\star$.
 Thus, the Vi\`{e}te map defined above restricts to a map 
$\V : \cC_A \rightarrow C(\C^\star,d)$ into the configuration space of $d$ points in $\C^\star$. We denote by 
$$\mu_A^\star:\pi_1(\cC_A) \rightarrow \pi_1\big(C(\C^\star,d)\big)$$ 
the induced map between fundamental groups and refer to it as the \emph{braid monodromy map} relative to $A$. The fundamental group of the configuration space $C(\C^\star,d)$ is the $d$-stranded \emph{surface braid group} on $\C^\star$ that we denote $B_d^\star$. We will also consider the map $\mu_A$ obtained by composition of the map $\mu_A^\star$ with the surjection $B_d^\star\rightarrow B_d$ induced by the inclusion $\C^\star \hookrightarrow \C$.

Following \cite{E18}, the support set $A \subset \N$ is said to be \emph{reduced} if the smallest sublattice of $\Z$ containing $A$ is $\Z$ itself. The set $A$ is said to be \emph{non-reduced} otherwise.

\begin{Theorem}\label{thm:reduced}
For any reduced set $A \subset \N$, the map $\mu_A^\star : \pi_1(\cC_A) \rightarrow B_d^\star$ is surjective. 
\end{Theorem}

For a non-reduced support set $B$ with largest element $\delta$, the map $\mu_{B}^\star$ is not surjective. Indeed, consider the reduced support $A:=B/b$ with largest element $d:=\delta/b$, where $b:=\gcd(B)$. Since every polynomial $q(x)\in \cC_B$ can be written as $p(x^b)$ for a unique $p(x)\in \cC_A$, we have the isomorphisms $\cC_A \simeq \cC_B$ and $\pi_1(\cC_A) \simeq \pi_1(\cC_B)$. The covering $f: x\mapsto x^b$ from $\C^\star$ to itself induces the pullback $f^*:B_d^\star \rightarrow B_{\delta}^\star$ satisfying $\mu^\star_{B}(\gamma) = f^*\big(\mu^\star_{ A}(\gamma)\big)$ for any element $\gamma \in \pi_1(\cC_{B}) \simeq \pi_1(\cC_{A})$. Together with Theorem \ref{thm:reduced}, the latter property implies the following.

\begin{Corollary}\label{cor:nonreduced}
Let $B \subset \N$ be a non-reduced support and let $d$, $\delta$ and $f$ be defined as above. Then, the image of the map $\mu^\star_{B} : \pi_1(\cC_{B}) \rightarrow B_{\delta}^\star$ is the subgroup  $f^*(B_d^\star)\subset  B_{\delta}^\star$.
\end{Corollary}

The image of $\mu^\star_{B}$ is isomorphic to $B_d^\star$ since $f^*$ is injective. However, the image of $\mu_B$, that is the projection of $f^*(B_d^\star)\subset  B_{\delta}^\star$ to $B_\delta$, is isomorphic to the wreath product of $\Z/b\Z$ with $B_d$, see the end of Section \ref{sec:braid}. 

\begin{Remark}
$\mathbf{1.}$ Corollary \ref{cor:nonreduced} justifies the consideration of braids in $\C^\star$ rather than in $\C$. Indeed, the computation of $\im(\mu_B)$ does not follow from the surjectivity of $\mu_A$, unlike for the case of $\mu_B^\star$ and $\mu_A^\star$, because there is no natural map $B_d\to B_\delta$ similar to $B_d^\star\to B_\delta^\star$. This phenomenon falls under the principle that the monodromy of an enumerative problem $P_B$ that is a covering of another enumerative problem $P_A$ can not be computed from the monodromy of the problem $P_A$ a priori. In general, more information is required on $P_A$, as illustrated for instance in \cite{AL}.

$\mathbf{2.}$ This in particular proves the theorem from \cite{AL} that the Galois group of the indeterminate polynomial with the non-reduced support $B$ equals the wreath product of $\Z/b\Z$ with $\Sym_d$. It would be interesting to extend results of this paper to multivariate polynomials so that they cover the respective results of \cite{AL} on Galois groups of systems of polynomial equations of several variables.
\end{Remark}

Let us come back to the case of a reduced support $A:=\{a_0,\cdots,a_n\}$ with $a_0=0$ and $a_n=d$. The set $\cC_A$ is the complement to a Zariski-closed subset of $\C^{d+1}$ that consists of the $A$-discriminant 
(see \cite{GKZ}) together with the first and last coordinate hyperplanes $\{c_0$=$0\}$ and $\{c_n$=$0\}$ in $\C^{d+1}$. While the monodromy obtained from small loops around the latter coordinate hyperplanes is easy to identify,  there might be situations when one is interested in the contribution of the remaining component of $\C^{d+1}\setminus \cC_A$. To this regard, we prove the following.

\begin{Theorem}\label{thm:contractible}
For a reduced support $A \subset \N$, the composition of $\pi_1(\cC_A\cap\{c_0$=$c_n$=$1\})\rightarrow \pi_1(\cC_A)$ with $\mu_A$ is surjective. In particular, if $K\subset \pi_1(\cC_A)$ denotes the kernel of the map induced by the inclusion $\cC_A\hookrightarrow\C^A \setminus\{c_0c_n$=$0\}$, then we have that $\mu_A(K)=B_d$.
\end{Theorem}

Theorem \ref{thm:contractible} has useful application to collections of fewnomials depending on a common set of parameters. Such situations appear naturally in the computation of the monodromy of Severi varieties of toric surfaces, see \cite{L19}. Consider a finite collection $A_1,\cdots, A_m\subset \N$, $\, 0\in A_j$, of reduced supports and an affine linear map 
$$L=(L_1,\ldots,L_m): \C^k \rightarrow \C^{A_1}\times \cdots \times \C^{A_m}.$$
For each index $j$, we denote by $\cD_j \subset \C^k$ (repectively $\mathcal{H}_j\subset \C^k$) the pullback under $L_j$ of the $A_j$-discriminant (respectively of $\{c_{0}c_{n_j}$=$0\}\subset \C^{A_j}$, where $n_j$ is the index of the largest element $d_j$ in $A_j$).

\begin{Definition}\label{def:generic}
The affine linear map $L$ is said to be \emph{generic} with respect to the collection $A_1,\cdots, A_m$ if the following holds:

-- the $m$ subsets $\mathcal{D}_j\subset \C^k$ are reduced hypersurfaces,

-- the $m$ subsets $\mathcal{H}_j \subset \C^k$ have positive codimension,

-- no $\mathcal{D}_j$ has a common irreducible component with another $\mathcal{D}_i$ or with any $\mathcal{H}_i$.
\end{Definition}

Observe that $\mathcal{D}_j$ is always reduced, irreducible and not contained in any of the $\mathcal{H}_i$ whenever $k\geqslant 2$. Observe also that we do not impose any restriction on the relative position of the $\mathcal{H}_j$-s and that the $\mathcal{H}_j$-s are allowed to be non-reduced. 

If we denote $U:=L^{-1}(\cC_{A_1}\times \cdots \times \cC_{A_m})$, then we can define the map 
\[ \mu : \pi_1(U) \rightarrow B_{d_1}\times \cdots \times B_{d_m} \]
where the $j^{th}$ factor of $\mu$ is the composition of the map $L_j^*:\pi_1(U) \rightarrow \pi_1(\cC_{A_j})$ with the map $\mu_{A_j}$.

\begin{Corollary}\label{cor:system}
If the linear map $L$ is generic in the sense of Definition \ref{def:generic}, then $\mu$ is surjective.
\end{Corollary}

\begin{proof}
Observe first that since $L$ is generic with respect to $A_1,\cdots,A_m$, there always exists an affine linear map $\iota:\C\rightarrow \C^k$ such that $\widetilde L:=L\circ\iota$ is generic. If we denote $\widetilde U:=\iota^{-1}(U)$ and $\tilde\mu:\pi_1(\widetilde U) \rightarrow B_{d_1}\times \cdots \times B_{d_m}$ the corresponding monodromy map, it suffices to show that $\tilde \mu$ is surjective since it factorises through $\mu$. In other words, we can assume that $k=1$.

Let us assume further that $L_j(\C)\subset \{c_{0}$=$c_{n_j}$=$1\}\subset \C^{A_j}$ for all $j$. In particular, all the $\mathcal{H}_j$-s are empty, the line $L_j(\C)$ intersects transversely the $A_j$-discriminant for any $j$ and the $\cD_j$-s are pairwise disjoint. By \cite[Th\'{e}or\`{e}me]{Ch}, the latter transversality property implies that $\pi_1(\C\setminus \cD_j)\rightarrow \pi_1(\cC_{A_j}\cap\{c_{0}$=$c_{n_j}$=$1\})$ is surjective. It follows from Theorem \ref{thm:contractible} that $\pi_1(\C\setminus \cD_j)\rightarrow B_{d_j}$ is also surjective. Since $U=\C \setminus\big(\cup_j \cD_j\big)$ and the fact that the $\cD_j$-s are pairwise disjoint, the natural map $\pi_1(U)\rightarrow\pi_1(\C\setminus \cD_1)\times \cdots \times \pi_1(\C\setminus \cD_m)$ is also surjective. Therefore, the map $\mu: \pi_1(U)\rightarrow B_{d_1}\times \cdots\times b_{d_m}$ is surjective too.

Let now $L:\C\rightarrow \C^{A_1}\times\cdots\times\C^{A_m}$ be any generic map. Then, there is a continuous one-parameter family $\{L_t\}_{0\leqslant t \leqslant 1}$ of such maps joining $L=L_0$ to a map $L_1$ as in the above paragraph. Let us put an index $t$ to all piece of notation corresponding to the map $L_t$. Then, there exists a continuous one-parameter family of discs $V_t\subset \C$ such that $V_t$ contains $\cup_j \cD_{j,t}$ and is disjoint from $\cup_j \mathcal{H}_{j,t}$. In particular, the disc $V_t$ determines the subgroup $G_t$ of  $\pi_1(U_t)$ of all loops in $U_t$ supported in $V_t$. Clearly, the group $\mu_t(G_t)$ is independent of $t$. Since $\cup_j \mathcal{H}_{j,1}$ is empty, we have that $G_1=\pi_1(U_1)$ and therefore that $\mu_t(G_t)=\mu_1(G_1)=\im(\mu_1)=B_{d_1}\times \cdots\times b_{d_m}$ by the previous paragraph.
\end{proof}
\medskip

\subsection{Techniques and perspectives}
Let us now comment on the proofs of  Theorems \ref{thm:reduced} and \ref{thm:contractible}. A possible approach is to use Zariski theorems such as \cite[Th\'{e}or\`{e}me]{Ch}. Indeed, the configuration space $\cC_A$ is an iteration of hyperplane sections of the space $\cC_d$. However, the latter hyperplane sections are not generic and the fundamental groups $\pi_1(\cC_A)$ and $\pi_1(\cC_d)$ are certainly not isomorphic in general. In particular, we cannot apply \cite[Th\'{e}or\`{e}me]{Ch}. At the current stage of our investigations, we are only interested in the surjectivity of  maps between fundamental groups induced by restrictions to hyperplane sections. Therefore, the statement  \cite[Theorem $2.5$]{Be} could be used, provided that its assumptions are satisfied in the present situation.

Instead, we use elementary considerations from tropical geometry to construct explicit elements in the image of $\mu_A^\star$ in the particular case $\#A=3$. We work out the general case by specialisation to the case of trinomials and by working out handy relations in the relevant braids groups, see Lemma \ref{lem:swap}. These relations lead to an analogue of the Euclidean algorithm which may be interesting in its own, see Proposition \ref{prop:euclide}.

In addition to providing explicit elements in the image of the braid monodromy, the tropical techniques developed in this paper have the advantage of generalising to families of univariate polynomials
\[ p(z,\underline{w})=c_0(\underline{w})+ c_1(\underline{w}) z^{a_1}+ \cdots + c_n(\underline{w}) z^{a_n} \]
whose coefficients are polynomials in $\underline{w}:=(w_1,\cdots,w_k)$ of arbitrary degrees, situation in which the Zariski theorems mentioned above do not apply. In particular, these techniques can be used to determine the Galois groups of  square systems 
\[ f_1=f_2=0\]
of bivariate polynomials supported on reducible tuples  $A_1,A_2 \subset \Z^2$, see \cite[Definition $1.3$]{E18}.

\subsection{Organisation of the paper} In Section \ref{sec:settings}, we recall standard facts about braid groups and the basics of tropical geometry that we use in Section \ref{sec:trinomials} to deal with the particular case of reduced trinomials. In Section \ref{sec:proofs}, we  use the case of trinomials and the aforementioned Euclidean algorithm to prove Theorems \ref{thm:reduced} and \ref{thm:contractible}.
\medskip

\textbf{Acknowledgements.} The authors are grateful to A. Libgober and M. L{\"o}nne for providing references on the subject and to S.Y. Orevkov for helpful discussion on braid groups.The authors are also grateful to an anonymous referee for valuable comments. The first author was supported by the Russian Science Foundation grant, project 16-11-10316.

\section{Settings}\label{sec:settings}

\subsection{Surface braid groups}\label{sec:braid}

Recall that $C(\C^\star,d)$ denotes the \emph{configuration space} of $d$ distinct points in $\C^\star$. Its fundamental group is denoted by $B_d^\star$ and referred to as the $d$-stranded \emph{surface braid group} on $\C^\star$. We refer to \cite[Ch.$9$]{farbmarg} for the theory on surface braid groups.

Define a \emph{$d$-stranded geometric braid} on $\C^\star$ to be a topological $1$-fold $b\subset S^1 \times \C^\star$ such that the projection $\pi: b\rightarrow S^1$ is a degree $d$ covering. The map $S^1\rightarrow C(\C^\star,d)$ defined by $\theta \mapsto \pi^{-1}(\theta)$ defines in turn an element in $B_d^\star$.
Two geometric braids define the same element in $B_d^\star$ if and only if there exists an isotopy of $d$-stranded geometric braids between them.

In order to perform computations in $B_d^\star$, we will represent braids by diagrams. To that purpose, observe that any braid $[b]\in B_d^\star$ can be represented by a geometric braid $b\subset S^1 \times \C^\star$ with projection $\pi: b\rightarrow S^1$ such that :

-- the set $\pi^{-1}(1)$ is evenly distributed on the circle $S^1 \subset \C^\star$,

-- the only singularities of the projection of $b$ under $\pi \times \arg$ are transverse self-intersection points (here, the function $\arg: \C^\star \rightarrow S^1$ is defined by $x\mapsto \frac{x}{\vert x \vert}$).

We represent geometric braids with the above properties using the projection $\rho$ given by the composition of $ \pi \times \arg$ with the projection of $S^1\times S^1$ to a square fundamental domain $[a, a+2\pi[ \times [0, 2\pi[$ for any choice $a \in \R$. When the projection of two strands of $b$ under $\rho$ cross each other, we picture the strand with the largest modulus in $\C^\star$ as passing underneath the other strand. Eventually, we use blue-shaded areas to represent a sequence of neighbouring strands that are parallel to each other. 

\begin{figure}[h]
\centering
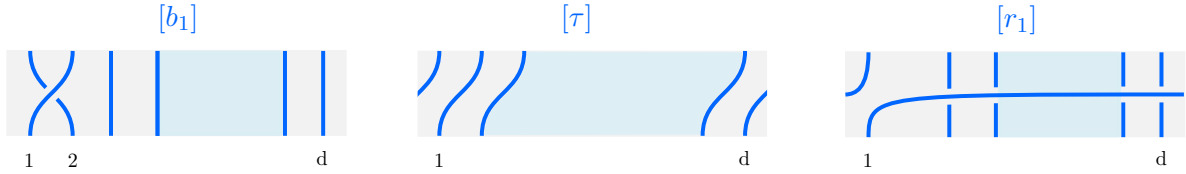
\caption{The elements $[b_1]$, $[\tau]$ and $[r_1]$ in $B_d^\star$.}
\label{fig:braidgenerators}
\end{figure}

Let us now introduce useful elements of the surface braid group $B_d^\star$. First, we define the elements $[b_1]$, $[\tau]$ and $[r_1]$ according to Figure \ref{fig:braidgenerators}. From these elements, we define
\[ [b_{j+1}] :=[\tau]^j \circ [b_1] \circ [\tau]^{-j} \; \text{ and } [r_{j+1}] :=[\tau]^j \circ [r_1] \circ [\tau]^{-j}\]
for $j \in \{1,\cdots,d-1\}$. We have the following classical result.

\begin{Lemma}\label{lem:generatorArtin}
The projection of the elements $[b_1], \cdots, [b_{d-1}] \in B_d^\star$ to $B_d$ generate $B_d$.
\end{Lemma}

\begin{proof}
The projection of these elements are the elements $\sigma_i$ used by E. Artin to give the presentation \cite[Theorem 16]{Ar} of $B_d$.
\end{proof}

A $d$-stranded braid is \emph{pure} if one (and then any) of its representative has $d$ connected components. Pure braids form the subgroup $PB_d^\star$ in $B_d^\star$. We order the connected components of a pure braid by ordering their starting points from left to right in the corresponding diagram (this ordering  depends on the choice of the fundamental domain $[a, a+2\pi[ \times [0, 2\pi[$). This ordering allows to define the map
\[
\begin{array}{rcl}
\ind \;  : \;  PB_d^\star &\rightarrow & \Z^d \\
\left[\beta\right] \quad & \mapsto & \frac{1}{2i\pi} \big( \int_{\beta_1}\frac{dx}{x}, \cdots,\int_{\beta_d}\frac{dx}{x} \big)
\end{array}
\]
where $\beta_j$ is the $j^{th}$ component of any representative of $[\beta]$.

\begin{Lemma}\label{lem:generatorsbraidgroup}
The surface braid group $B_d^\star$ is generated by any of the following sets:

-- the set $\{[\tau], [b_1], \cdots, [b_{d}], [r_1], \cdots,[r_d]\}$,

-- the set $\{[\tau], [b_1], [r_1]\}$,

-- the set $\{[\tau], [b_1]\}$.
\end{Lemma}

\begin{proof}
By definition, the first set is generated by the second. The relation $[r_1]=[\tau]\circ [b_{d-1}]\circ \cdots \circ [b_1]$, that is easily check using the diagrams of Figure \ref{fig:braidgenerators}, implies that the third set generates the second one. It remains to show that the first set is a generating set. 

To see this, observe first that for any element $[b] \in B_d^\star$, there is an element $[b']$ in the group generated by $[b_1], \cdots, [b_{d-1}]$ such that $[b']\circ[b]$ is a pure braid. Thus, there is no loss of generality in assuming that $[b]$ is pure. Now, for any pure braid $[b]$, there is an element $[b']$ in the group generated by $ [r_1], \cdots, [r_{d}]$ such that $\ind([b']\circ[b])=(0,\cdots,0)$.Thus, there is no loss of generality in assuming that $[b]$ is pure and such that $\ind([b])=(0,\cdots,0)$. We now want to argue that this is possible to unravel such a braid with the only elements $ [b_1], \cdots, [b_{d}]$. To see this, observe that there exists an integer $m$ such that every string in the diagram of the pullback $[b']$ of the braid $[b]$ by the covering $x\mapsto x^m$ fits inside a single fundamental domain. This follows from the fact that $\ind([b])=(0,\cdots,0)$. We know from Lemma \ref{lem:generatorArtin} that the braids $[b_1], \cdots, [b_{md-1}]$ suffice to unravel the braid $[b']\in B_{md}^\star$, that is we can write $[b']= [b_{j'_1}]\circ \cdots  [b_{j'_n}]$. Clearly, we can write $[b]= [b_{j_1}]\circ \cdots  [b_{j_n}]$ where $j_i\in \{1,\cdots,d\}$ is the reduction of $j'_i$ modulo $d$.
\end{proof}

In the course of the paper, we will construct braids by means of tropical geometry. Unfortunately, the objects we will obtain are not as nice as geometric braids. However, we argue below that these objects lead to well defined elements in  $B_d^\star$ in a rather canonical way.

\begin{Definition}\label{def:coarsebraid}
A $d$-stranded \emph{coarse braid} in $\C^\star$ is a topological 1-fold $b'\subset S^1 \times \C^\star$ with associated projection $\pi': b' \rightarrow S^1$ such that $\pi'$ is a degree $d$ covering outside of a finite set $T\subset S^1$ and such that the preimage by $\pi'$ of any connected arc in $S^1$ with endpoints $\alpha, \beta \in S^1\setminus T$ consists of $d$ connected arcs in $b'$ connecting $\pi'^{-1}(\alpha)$ to $\pi'^{-1}(\beta)$.
\end{Definition}

\begin{Lemma}\label{lem:braid}
For any coarse braid $b'\subset S^1 \times \C^\star$, there exists a family of geometric braids $b_\varepsilon$ continuous in $\varepsilon>0$ whose limit in the Hausdorff distance is $b'$. Moreover, the class $[b_\varepsilon] \in B_d^\star$ depends neither on $\varepsilon$ nor on the choice of the family $b_\varepsilon$.
\end{Lemma}

\begin{Definition}\label{def:classcoarsebraid}
For any $d$-stranded coarse braid $b'\subset S^1 \times \C^\star$, we define \emph{the class of the coarse braid} $b'$ to be the element $[b_\varepsilon] \in B_d^\star$ for any family $b_\varepsilon$ as in Lemma \ref{lem:braid}. We denote this class by $[b']\in B_d^\star$.
\end{Definition}

\begin{proof}[Proof of Lemma \ref{lem:braid}]
Let us first show the existence of a family $b_\varepsilon$ as in the statement. Let us fix $\delta>0$ arbitrarily small. We prescribe that $b_\varepsilon$ is constant and  coincides with $b'$ on the preimage by $\pi'$ of the complement of the $\delta$-neighbourhood of $T\subset S^1$. Now, for any $\theta\in T$, the preimage  $(\pi')^{-1}\big( ]\theta-\delta, \theta+\delta[\big)$ consists of $d$ disjoints connected arcs, some of which  are not transverse to the fiber $(\pi')^{-1}(\theta)$.  Obviously, there is exists a  deformation of each such arc satisfying the following:

-- the deformation is continuous in the parameter $\varepsilon$;

-- the arcs have fixed endpoints; 

-- the arcs are transverse to the projection $\pi'$ for any $\varepsilon>0$.\\
The latter deformation provides us the sought family $b_\varepsilon$.

Finally, if $\textbf{b}_\varepsilon$ is another family as in the statement, 
it is clear that $b_\varepsilon$ is 
homotopic to  $\textbf{b}_\varepsilon$ along a family of geometric braids for small $\varepsilon$. In particular, the class $[b_\varepsilon]$ does not depend on the choice of the family $b_\varepsilon$ and obviously, it does not depend on $\varepsilon$ either.
\end{proof}

We conclude this section with the description of the isomorphism between the projection of $f^*(B^\star_d)$ in $B_d$ and the wreath product of $\Z/b\Z$ with $B_\delta$ in the context of Corollary \ref{cor:nonreduced}. Recall briefly the definition of the wreath product $\Z/b\Z \wr B_d$. The group $B_d$ acts on $\{1,\cdots,d\}$ by permutation. The latter action extends naturally to the group $K:=\{1,\cdots,d\}^{\Z/b\Z}$. The group $\Z/b\Z \wr B_d$ is defined as the semidirect product $K \rtimes B_d$ with respect to the latter action. In other words, elements in $\Z/b\Z \wr B_d$ are braids in $B_d$ each strand of which is decorated with an element in $\Z/b\Z$. The multiplication of two such elements is the one of $B_d$ at the level of braids and to a strand of the product we associate the sum of the elements in $\Z/b\Z$ decorating the two concatenated strands.

Recall now that $\delta=bd$ and that the elements of $f^*(B^\star_d)\subset B_\delta^\star$ are the braids globally invariant under multiplication by the $b^{th}$ roots of unity. There is not loss of generality in assuming that the braids of $B^\star_d$ are based at $\{1,\cdots,d\}$ and therefore that the braids of $f^*(B^\star_d)$ are based at the $b^{th}$ roots of $1,\cdots,d$. In particular, the $\delta$ strands of a braid $[b]\in f^*(B^\star_d)$ are divided into $d$ tuples of $b$ strands, namely the strands based at a $b^{th}$ root of $j$ for any $j\in \{1,\cdots,d\}$. If the strand based at $j\in \C$ ends at $j'e^{2i\pi k/b}$, then the strand based at $je^{2i\pi \ell/b}$ ends at $j'e^{2i\pi (\ell+k)/b}$, since each tuple is globally invariant under multiplication by $e^{2i\pi/b}$. Therefore,  each of the $d$ tuples defines an element in $k\in \Z/b\Z$. If $[b]=f^*([b'])$ then, each of the $d$ tuples of $[b]$ contains exactly one strand of $[b']$. Therefore, we can define the map $\varphi:f^*(B^\star_d)\rightarrow \Z/b\Z \wr B_d$ sending $[b]$ to the projection of $[b']$ in $B_d$ such that each strand $\beta$ of the latter projection is decorated with the element of $\Z/b\Z$ defined by the $b$-tuple of strands in $[b]$ containing $\beta$.

The map $\varphi$ is obviously surjective. It remains to show that the kernel of $\varphi$ coincide with the kernel of the projection $\pi:f^*(B^\star_d)\rightarrow B_\delta$. It is an easy exercise to check that these kernels in $f^*(B^\star_d)\simeq B^\star_d$ both coincide with $\left\langle r_1^b, \cdots, r_d^b\right\rangle\subset B^\star_d$.

\subsection{The phase-tropical line}\label{sec:pthyper}
Let $z$ and  $w$ be the coordinates functions on $\C^2$ and denote by $\cH$ the line
\[ \cH:=\left\lbrace (z,w)\in \C^2 \, \big| \, 1+z+w = 0\right\rbrace.\] 
The \emph{phase-tropical line} $L\subset (\C^\star)^2$ is the Hausdorff limit of the sets $H_t\big(\cH \cap (\C^\star)^2\big)$ where
\[
\begin{array}{rcl}
H_t \quad  : \quad (\C^\star)^2 & \rightarrow & (\C^\star)^2\\
(z,w) &\mapsto& \Big( \frac{z}{\vert z\vert } \vert z\vert^{\frac{1}{\log(t)}},
\frac{w}{\vert w\vert} \vert w\vert^{\frac{1}{\log(t)}}\Big) 
\end{array}
\]
as $t$ goes to $+\infty$. The set $L$ turns out to be homeomorphic to $\cH\simeq \cp{1} \setminus \{0,1,\infty\}$ and can be described through its projections under the maps\\
\begin{minipage}{6,5cm}
\vspace{0,2cm}
\[
\begin{array}{rcl}
\Arg : (\C^\star)^2 &\rightarrow& (S^1)^2\\
(z,w) & \mapsto & \big( \frac{z}{\vert z \vert}, \frac{w}{\vert w \vert} \big)
\end{array}
\]
\smallskip
\end{minipage}
and
\begin{minipage}{7,5cm}
\vspace{0,2cm}
\[
\begin{array}{rcl}
\Log : (\C^\star)^2 &\rightarrow& \R^2 \\
(z,w) & \mapsto & \big( \log(\vert z\vert ), \log(\vert w \vert) \big)
\end{array} .
\]
\smallskip
\end{minipage}

The set $\Log(L)$ is the union of the three rays $r_z:=(-1,0)\cdot\R_{\geqslant 0}$, $r_w:=(0,-1)\cdot\R_{\geqslant 0}$ and $r_\infty:=(1,1)\cdot\R_{\geqslant 0}$ merging at the origin $0\in \R^2$. For any point $p$ in the relative interior of $r_z$ (respectively $r_w$, respectively $r_\infty$), the set $\Arg\big(L\cap \Log^{-1}(p)\big) \subset (S^1)^2$ is the geodesic with the same slope as $r_z$ (respectively $r_w$, respectively $r_\infty$) passing through the point $(-1,-1)$ (respectively $(-1,-1)$, respectively $(-1,1)$). Eventually, the set $\Arg\big(L\cap \Log^{-1}(0)\big) \subset (S^1)^2$ is the closure of $\Arg(\cH)$. The latter set is the union of the two blue triangles bounded by the three geodesics $g_z$, $g_w$ and $g_\infty$ pictured in Figure \ref{fig:ptline} (left). In the same figure, we illustrate the description of $L$ in terms of the projections $\Log$ and $\Arg$.

\begin{figure}[h]
\centering
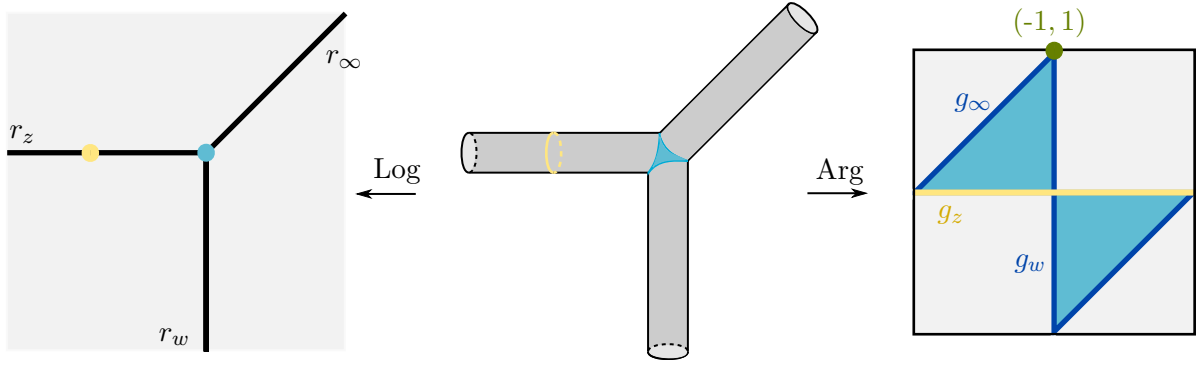
\caption{The phase-tropical line $L$ (in the middle) together with its projections under $\Log$ and $\Arg$. The lifts of the two coloured points in $\Log(L)$ are pictured  with the same colors in $L$ as well as their projection under $\Arg$. In particular, the set $\overline{\Arg(\cH)}$ is the union of the two blue triangles in $(S^1)^2$. Observe that we chose $(-1,-1)$ as the center of the fundamental domain for $(S^1)^2$.}
\label{fig:ptline}
\end{figure}

It will also be convenient to observe that the phase-tropical line $L$ is piecewise algebraic. Indeed, the restriction of $L$ to $\Log^{-1}( r_z \setminus 0)$ (respectively $\Log^{-1}( r_w \setminus 0)$, respectively $\Log^{-1}( r_\infty \setminus 0)$) is the algebraic curve $z=1$ (respectively $w=1$, respectively $z=-w$).

\section{Trinomials}\label{sec:trinomials}

Throughout this section, the support $A$ will be of the form $A:=\{0, p, d\}$ where $p$ and $d$ are coprime. 

\subsection{Tropicalisation}

In order to determine the image of $\mu^\star_A$, we will construct braids using tropical methods. To do so, we first reformulate our problem in a way that is suitable for tropicalisation.
 
In this section, we restrict to trinomials of the form $p(x)=1+c_1x^{p}+c_2x^{d}$ in $\cC_A$, where $c_0$, $c_1$, $c_2$ are the coordinates on $\C^A$. We denote simply by  $c:=(c_1,c_2)\in\C^2$ the corresponding coefficient vector and by $\cS_c:= \left\lbrace x\in \C^\star \, \big| \, p(x)=0 \right\rbrace$ the corresponding \emph{set of solutions}. The  coefficient vector $c$ also defines a map
\begin{equation}
\begin{array}{rcl}
\phi_c : \C^\ast & \rightarrow & \C^2\\
x &\mapsto & (c_1x^{p},c_2x ^{d})
\end{array}
\end{equation}
relative to $A$.
Thus, the set $\cS_c$
equals $\phi_c^{-1}(\cH)$ where $\cH$ is the line defined in Section  \ref{sec:pthyper}. Recall from the same section that $L\subset (\C^\star)^2$ denotes the phase-tropical line. By analogy, we define the \emph{set of tropical solutions} $\st_c:=\phi_c^{-1}(L) \subset \C^\star$ and the \emph{space of tropical conditions} $\ct_A \subset \C^2$ to be the set of vectors $c\in \C^2$ such that $\st_c$ consists of $d$ connected components. The following observation will be useful.

\begin{Lemma}\label{lem:tropdiscr}
A vector $c\in \C^2$ is in the complement of $\ct_A$ if and only if the following conditions are satisfied :

-- the line $\Log\big(\phi_c(\C^\star)\big)$ passes through the origin $0 \in \R^2$; 

-- the geodesic $\Arg\big(\phi_c(\C^\star)\big)$ passes through the point $(-1,1) \in (S^1)^2$.\\
Concisely, the set $\C^2 \setminus \ct_A$ is the algebraic curve $(-c_1)^d=(c_2)^p$.
\end{Lemma}
\begin{proof}
If the line $\Log\big(\phi_c(\C^\star)\big)$ does not pass through $0\in \R^2$, then the latter line intersects $\Log(L)$ transversely in either one or two points. The respective fibers in $\phi_c(\C^\star)$ and $L$ over these points consist of geodesics intersecting transversely in the corresponding argument tori  in $d$ points in total. We deduce that the first condition is necessary.

Assume therefore that  $\Log\big(\phi_c(\C^\star)\big)$ passes through $0$. Observe that the origin is the only intersection point between $\Log\big(\phi_c(\C^\star)\big)$ and $\Log(L)$ in that case. The connected components of the set of tropical solutions $\st_c$ are then in correspondence with the connected components of $\Arg\big(\phi_c(\C^\star)\big) \cap \Arg(\cH)$. The geodesic $\Arg\big(\phi_c(\C^\star)\big)$ has slope $(p,d)$, hence it intersects the geodesic $g_z$ of Figure \ref{fig:ptline} in $d$ distinct points $x_1,\cdots,x_d$ (recall that $p$ and $d$ are coprime). Each of the latter point is the endpoint of a connected component of $\Arg\big(\phi_c(\C^\star)\big) \cap \Arg(\cH)$. If $\Arg\big(\phi_c(\C^\star)\big)$ does not pass through $(-1,1)$, the other endpoint lies on exactly one of the two remaining geodesics $g_w$ and $g_\infty$. In particular, the other endpoint is away from the yellow geodesic and each of the $x_j$-s defines a distinct connect component of $\Arg\big(\phi_c(\C^\star)\big) \cap \Arg(\cH)$. Thus, the second condition in the statement is also necessary.

By the same argument, we observe that $\Arg\big(\phi_c(\C^\star)\big) \cap \Arg(\cH)$ consists of exactly $d-1$ connected components when $\Arg\big(\phi_c(\C^\star)\big)$ passes through $(-1,1)$. It follows that the two conditions are sufficient.

To conclude, the two conditions are equivalent to the fact that $\phi_c(\C^\star)$ passes through the point $(-1,1)\in \C^2$, which in turn is equivalent to the possibility of writing $c=(-x^p,x^d)$ for some $x \in \C^\star$. We deduce that the set $\C^2 \setminus \ct_A$ is the algebraic curve $(-c_1)^d=(c_2)^p$.
\end{proof}

\begin{Remark}\label{rem:Lib&stable}
$\mathbf{1.}$ The above lemma is the tropical counterpart to \cite[Theorem B]{Lib90}. Here, we recover the fact that the fundamental group of $\ct_A$ is the group of the $(p,d)$-torus knot.

$\mathbf{2.}$ The algebraic curve $(-c_1)^d=(c_2)^p$ is globally fixed by $H_t$ for any $t>0$, and so is its complement $\ct_A$. Indeed, the self-diffeomorphism $H_t$ corresponds to the rescaling by $\frac{1}{\log(t)}$ on the first factor of the logarithmic coordinates $\Log \times \Arg$ and the projection of $(-c_1)^d=(c_2)^p$ under $\Log$ is a line passing through $0\in \R^2$ and is therefore preserved under the latter rescaling. 
\end{Remark}

We are now ready to show that coarse braids obtained by tropical construction lead to braids in the image of $\mu^\star_A$.

\begin{Proposition}\label{prop:correspondence}
Let $\gamma : S^1 \rightarrow \ct_A$ be a loop such that the set $b':= \left\lbrace (\theta, p) \in S^1\times \C^\star \, \big|  \, p \in  \st_{\gamma(\theta)}\right\rbrace$ is a coarse braid. 
Then, there exists $\ell$ in $\pi_1(\cC_A)$ such that $\mu_A^\star(\ell)=[b'] \in B_d^\star$ (see Definitions \ref{def:coarsebraid} and \ref{def:classcoarsebraid}). Moreover, we can choose $\ell$ so that $\im(\ell)$ is contained in $\{c_0$=$c_2$=$1\}$ (respectively $\{c_0$=$c_1$=$1\}$) if $\im(\gamma)$ is. 
\end{Proposition}

\begin{proof}
The strategy is to construct an explicit one-parameter family of geometric braids as in Lemma \ref{lem:braid}. This family will be chosen in two different ways in order to fit with the  extra assumption that $\im(\gamma)$ might be contained in $\{c_0$=$c_2$=$1\}$  (respectively in $\{c_0$=$c_1$=$1\}$).

For any $c=(c_1,c_2) \in \C^2$ and any $t>0$, we define the element $c^t \in \C^2$ by 
$$\textstyle c^t:=\big(h_t\big(\frac{d}{p^{p/d}}\big)\cdot c_1,c_2\big) \; \big(\text{respectively } c^t:=\big(c_1,h_t\big(\frac{p}{d^{d/p}}\big)\cdot c_2\big) \big)$$
where $h_t$ is the self-diffeomorphism $h_t(x)=\frac{x}{\vert x\vert } \vert x\vert^{\frac{1}{\log(t)}}$ on $\C^\star$. 
For both definitions of $c^t$, we claim first that $H_t^{-1}(c^t)\in \cC_A$ for any any $t>0$ if $c\in \ct_A$ and second that $h_t\big(\cS_{H_t^{-1}(c^t)}\big)$ converges in Hausdorff distance to $\st_c$ when $t$ becomes arbitrarily large. If we denote $c_\theta:=\gamma(\theta)$ and $c_\theta^t$ constructed from $c_\theta$ as above, we deduce that for any $t>0$, the set 
\[b_t := \left\lbrace (\theta,q) \in S^1\times \C^\star \big| q \in h_t \big( \cS_{H_t^{-1}(c_\theta^t)}\big) \right\rbrace \]
is a geometric braid that can be made arbitrarily close to $b'$ for sufficiently large $t$. Now, observe that the geometric braids $b_t$ are isotopic to each other for any $t\geqslant e$ and that $b_e$ is the braid 
\[b_e = \left\lbrace (\theta,q) \in S^1\times \C^\star \big| q \in  \cS_{c_\theta^e} \right\rbrace.\]
Defining $\ell : S^1 \rightarrow \cC_A$ by $\ell(\theta)=c_\theta^e$, we deduce that that $[b_e]= \mu_A^\star(\gamma)$. Setting $\varepsilon:=1/t$, we conclude that the one-parameter family $b_\varepsilon$ fulfils the assumptions of Lemma \ref{lem:braid} and therefore that $\mu_A^\star(\ell)=[b']$. Moreover, it is clear from the construction that $\im(\ell)$ is contained in $\{c_0$=$c_2$=$1\}$ (respectively $\{c_0$=$c_1$=$1\}$) if $\im(\gamma)$ is.

It remains to prove the claims. We give the proof for the first definition of $c^t$ (the proof for the second definition is similar). For the first claim, observe that $\cC_A$ is the complement of the curve $($-$c_1/d)^d=(c_2/p)^p$. Indeed, we can parametrise the pairs $(c,x)$ such that the trinomial corresponding to $c\in\C^A$ is singular at $x$ by $x\mapsto \big(\frac{-d}{x^d(d-p)}, \frac{p}{x^p(d-p)}\big)$. Then, the vector $H_t^{-1}(c^t)$ is not in $\cC_A$ if and only if 
$$\textstyle \big(-\frac{h_t^{-1}(c_1)}{d}\cdot \frac{d}{p^{p/d}}\big)^d=\big(\frac{h_t^{-1}(c_2)}{p}\big)^p \; \Leftrightarrow \; \big(-h_t^{-1}(c_1)\big)^d=\big(h_t^{-1}(c_2)\big)^p$$
 if and only if $H_t^{-1}(c)$ is not in $\ct_A$, by Lemma \ref{lem:tropdiscr}. By Remark \ref{rem:Lib&stable}.2, the latter is equivalent to $c$ not lying in $\ct_A$. For the second claim, we have for any $c=(c_1,c_2)\in \C^2$ that 
\[
\begin{array}{rl}
\st_c \;:=\;\displaystyle \phi_c^{-1}\Big(\lim_{t\rightarrow \infty} H_t(\cH)\Big)& \supset \;
\displaystyle \lim_{t\rightarrow \infty} \phi_c^{-1}\big(H_t(\cH)\big) \\
& \\
& = \; \displaystyle \lim_{t\rightarrow \infty}\phi_c^{-1} \left( \left\lbrace  (z,w) \in (\C^\star)^2 \, \big| \, 1 + h_t^{-1}(z)+h_t^{-1}(w)=0\right\rbrace\right)\\ & \\
& = \; \displaystyle \lim_{t\rightarrow \infty}\left\lbrace  x \in \C^\star \, \big| \, 1 + h_t^{-1}(c_1x^p)+h_t^{-1}(c_2x^d)=0\right\rbrace\\ & \\
& =\; \displaystyle \lim_{t\rightarrow \infty}\left\lbrace  x \in \C^\star \, \big| \, 1 + h_t^{-1}(c_1)\big(h_t^{-1}(x)\big)^p+h_t^{-1}(c_2)\big(h_t^{-1}(x)\big)^d=0\right\rbrace\\ & \\
& = \displaystyle \lim_{t\rightarrow \infty}\; h_t \big( \phi_{H_t^{-1}(c)}(\cH) \big) \; := \; \lim_{t\rightarrow \infty}\; h_t \big( \cS_{H_t^{-1}(c)}\big)
\end{array}
\]
It implies that the set $h_t \big( \cS_{H_t^{-1}(c)} \big)$ can be made arbitrarily close to $\st_c$ for $t$ sufficiently large. Since $c^t$ converges to $c$ for large $t$, the claim follows.
\end{proof}

\subsection{Tropical construction of coarse braids}

Motivated by Proposition \ref{prop:correspondence}, we will construct explicit loops in the space of tropical conditions $\ct_A$ leading to coarse braids.

We will now construct the loop $\gamma : S^1 \rightarrow \ct_A$ giving rise to the sought coarse braid.
First, we fix the base-point $c^{(0)}$ of $\gamma$ to be of the form $c^{(0)}:=(e^{i(\pi-\varepsilon)-1},1)$ where $\varepsilon>0$ is arbitrarily small. In particular, the geodesic $\Arg\big(\phi_{c^{(0)}}(\C^\star)\big)$ of slope $(p,d)$ passes arbitrarily close to the point $(-1,1)$. The line $\Log\big(\phi_{c^{(0)}}(\C^\star)\big)$ passes through $(-1,0)$. We deduce from Lemma  \ref{lem:tropdiscr} that $c^{(0)}$ belongs to $\ct_A$. The itinerary of $\gamma$ is as follows.

\textbf{Day 1}: we start from $c^{(0)}$ and follow the path $x \mapsto c:= (e^{x+i(\pi-\varepsilon)},1)$ from $-1$ to $1$. While doing so, the geodesic $\Arg\big(\phi_{c}(\C^\star)\big)$ remains unchanged. In particular, we stay inside $\ct_A$. The line $\Log\big(\phi_{c}(\C^\star)\big)$ passes through $(x,0)$  and moves to the right as $x$ increases. In particular, the latter line passes through $ 0\in\R^2$ when $x=0$. We denote the corresponding point $c^{(\times)}:=(e^{i(\pi-\varepsilon)},1)$. We take our first break at the point $c^{(1)}:=(e^{1+i(\pi-\varepsilon)},1).$

\textbf{Day 2}: we start from $c^{(1)}$ and follow the path $\theta \mapsto c:=(e^{1+i\theta},1)$ from $\pi-\varepsilon$ to $\pi+\varepsilon$. In particular, we stay inside $\ct_A$. The geodesic $\Arg\big(\phi_{c}(\C^\star)\big)$ passes through $(e^{i\theta},1)$  and moves to the right as $\theta$ increases, crossing $(-1,1)$ on the way. We take our second break at the point $c^{(2)}:=(e^{i(\pi+\varepsilon)+1},1).$

\textbf{Day 3}: we start from $c^{(2)}$ and follow the path $x \mapsto c:=(e^{x+i(\pi+\varepsilon)},1)$ from $1$ to $-1$. While doing so, the geodesic $\Arg\big(\phi_{c}(\C^\star)\big)$ remains unchanged. In particular, we stay inside $\ct_A$. Following this path, the geodesic $\Log\big(\phi_{c}(\C^\star)\big)$ moves to the left as $x$ decreases.  In particular, the latter line passes through $ 0\in\R^2$ when $x=0$. We denote the corresponding point $c^{(+)}:=(e^{i(\pi+\varepsilon)},1)$. We take our third break at the point $c^{(3)}:=(e^{i(\pi+\varepsilon)-1},1).$

\textbf{Day 4}: we start from $c^{(3)}$ and follow the path $\theta \mapsto c:=(e^{i\theta-1},1)$ from $\pi+\varepsilon$ to $\pi-\varepsilon$. We recover the point $c^{(0)}$ which is our final destination. During this last day, we stayed inside $\ct_A$ since the line $\Log\big(\phi_{c}(\C^\star)\big)$ remained unchanged.

\begin{Proposition}\label{prop:trinomial}
For the loop $\gamma : S^1 \rightarrow \ct_A$ described above, the subset $b':= \left\lbrace (\theta, p)  \, \big|  \, p \in  \st_{\gamma(\theta)}\right\rbrace$ of $S^1\times \C^\star$ is a coarse braid. Moreover, the class $[b'] \in B_d^\star$ is the element $[b_1]^{-1}$.
\end{Proposition}

The statement above depends on the choice of fundamental domain we use to represent braids, see Section \ref{sec:braid}. We will make this choice precise in the course of the proof. The arguments in the proof below are rather geometric and fairly elementary. It might be helpful though to consider the concrete exemple pictured in Figure \ref{fig:tropbraid}. 

\begin{proof}
Denote the projection $\pi': b' \rightarrow S^1$ as at the end of Section \ref{sec:braid}.
Travelling along $\gamma$, the line $\Log\big(\phi_c(\C^\star)\big)$ passes through $0\in\R^2$ exactly twice, namely when $c\in \{c^{(\times)}, c^{(+)}\}$. For such $c$, the geodesic $\Log\big(\phi_c(\C^\star)\big)$ avoids the point $(-1,1)$ so that the set $\st_{c}$ consists of $d$ path connected components. Similarly, we observe that the preimage $\pi'^{-1}(U)$ of a small neighbourhood $U$ around any of $c\in \{c^{(\times)}, c^{(+)}\}$ consists of $d$ path connected components. For any other $c$ in $\im(\gamma)$, the set $\st_{c}$ consists of $d$ distinct points varying continuously. It follows then that $b'$ is a coarse braid.

To show that $[b']=[b_1]^{-1}$, we have the liberty to choose the real number $a$ determining the fundamental domain $[a,a+2\pi[ \times [0,2\pi[$ which we use to represent braids. We choose $a=-\frac{2\pi}{d}$. Then, we have that the $d$ points $\{e^{i2\pi k/d} \, \vert \, k\in \Z \}$ (to which $e^{ia}$ belongs) are mapped under $\Arg \circ \phi_{c^{(0)}}$ bijectively to the intersection of  $\Arg\big(\phi_{c^{(0)}}(\C^\star) \big)$ with the geodesic $\{\arg(w)=1\}$.
The image under $\Arg$ of the $d$ points of intersection $\phi_{c^{(0)}}(\C^\star) \cap L$ are evenly distributed on the geodesic $\Arg\big(\phi_{c^{(0)}}(\C^\star) \big)$ and the geodesic $g_z:=\{\arg(w)=-1\}$. It follows that there is exactly one point of $\arg(\st_{c^{(0)}})$ in each component of the complement to the collection of evenly distributed points $\{e^{i2\pi k/d} \, \vert \, k\in \Z \}$ on $S^1$. While travelling along $\gamma$ between $c^{(0)}$ and $c^{(\times)}$, the set $\arg(\st_{c})$ is constant until we reach $c^{(\times)}$. By the above arguments, there is exactly one component of $\arg(\st_{c^{(\times)}})$ in each component of $S^1 \setminus \{e^{i2\pi k/d} \, \vert \, k\in \Z \}$. By construction of $\gamma$ and the choice of $a$, the two components of $\arg(\st_{c^{(\times)}})$ in the arcs $[e^{ia}, e^{i(a+2\pi/d)}]$ and $[e^{ia+2\pi/d}, e^{i(a+4\pi/d)}]$ are arbitrarily close to each other while the distance between any other pair of adjacent components of $\arg(\st_{c^{(\times)}})$ is arbitrarily large relative to $\varepsilon$. Equivalently, the projection under $\arg$ of the first two strands of $b'$ between $c^{(0)}$ and $c^{(\times)}$ are arbitrarily close to each other while the projection of other strands are far apart. The projection of the latter strands is unchanged between $c^{(\times)}$ and $c^{(2)}$. Between $c^{(2)}$ and $c^{(3)}$, the projection under $\arg$ of the first two strands of $b'$ meet when $\theta=\pi$. The projection of the first strand under $\Log$ sits on the ray $r_\infty \subset \Log(L)$ while the projection of the second strand sits on the ray $r_w \subset \Log(L)$. This implies that locally, the first strand has larger modulus than the second. In the corresponding diagram of $b'$, the first strand passes under the second. The projection under $\arg$ of the remaining strands of $b'$ do not meet between $c^{(2)}$ and $c^{(3)}$ and it should be clear from now that there will not meet between $c^{(3)}$ and our return to $c^{(0)}$. Since the same is true for the first two strands, the result follows.
\end{proof}

Recall that we have coordinates $c_0$, $c_1$, $c_2$ on $\C^A$ corresponding respectively to the monomials $1$, $x^p$ and $x^d$. 

\begin{Lemma}\label{lem:trinomial}
For any $j \in \{1,\cdots,d\}$, there exists a loop $\ell_j : S^1 \rightarrow \ct_A$ 
such that the subset $b_j':= \left\lbrace (\theta, p)  \, \big|  \, p \in  \st_{\ell_j(\theta)}\right\rbrace$ of $S^1\times \C^\star$ is a coarse braid of class $[b_j] \in B_d^\star$ and such that $\im(\ell_j)\subset \{c_0$=$c_2$=$1\}$. Moreover, the loop $\ell : S^1 \rightarrow \ct_A$ given by $\ell(e^{i \theta})=(e^{1+i(\pi-\varepsilon)},e^{-i\theta})$ defines a geometric braid $b$ with class $[b]=[\tau]$.
\end{Lemma}

\begin{figure}[H]
\centering
\scalebox{1}{
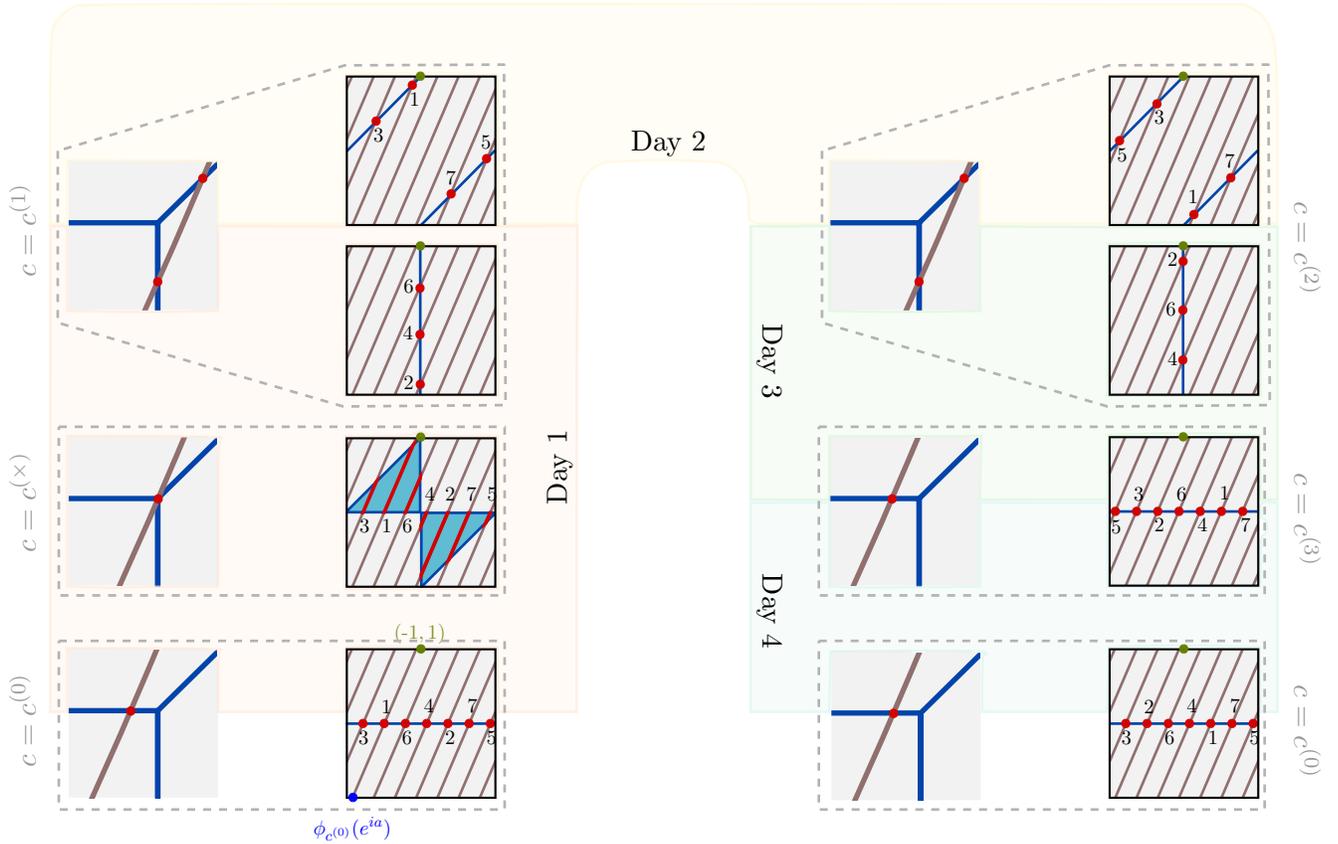}
\caption{The loop $\gamma$ for $A=\{0,3,7\}$. We represent the projection under $\Log$ (left) and $\Arg$ (right) of $L$ (blue) and $\phi_c(\C^\star)$ (brown). The projection of the intersection $L \cap \phi_c(\C^\star)$ is pictured in red in both $\R^2$ and $(S^1)^2$. The connected components of the red locus in $(S^1)^2$ are cyclically ordered along the brown geodesic, and ordered according to the choice of the parameter $a$ made in the proof of Proposition \ref{prop:trinomial}.}
\label{fig:tropbraid}
\end{figure}

\begin{proof}
For $j=1$, it suffices to consider the loop $\ell_1:=\gamma^{-1}$ of Proposition \ref{prop:trinomial}. For any other $j$, observe that the set of tropical solutions $\st_{c^{(0)}}$ coincides with $\st_{c_{(j)}}$ where $c_{(j)}:=(e^{i(\pi-\varepsilon)-1}e^{2i\pi \frac{j-1}{d}},1)$. Therefore,  we can also construct the loop $\gamma_j$ defined as $\gamma$ except that the base-point of $\gamma_j$ is $c_{(j)}$ instead of $c^{(0)}$. Consider the path $\rho_j:[0,1]\rightarrow \ct_A$ given by $$\rho_j(\theta)=(e^{i(\pi-\varepsilon)-1}e^{2i\pi\theta\frac{ (j-1)}{d}},1)$$ connecting $c^{(0)}$ to $c_{(j)}$. Then, the sought loop $\ell_j$ for $j \in \{2,\cdots, d\}$ is defined by $\ell_j :=(\rho_j)^{-1}\circ (\gamma_{j})^{-1}\circ \rho_j$. 
It follows from the construction that $\im(\ell_j)\subset \{c_0$=$c_2$=$1\}$ for any $j$.

Eventually, the set  $b:= \left\lbrace (\theta, p)  \, \big|  \, p \in  \st_{\ell(\theta)}\right\rbrace$ defined by $\ell$ is easily seen to be a geometric braid since $\Log\big(\phi_{\ell(\theta)}(\C^\star)\big)$ never passes through $0\in \R^2$. Each of the points in the set of tropical solutions $\st_{\ell(\theta)}$ travels clockwise along $S^1 \subset  \C$ with constant velocity $\frac{2\pi}{d}$ so that the diagram of $[b]$ is the same as the one of $[\tau]$.
\end{proof}

\begin{Corollary}\label{cor:trinomial}
The map $\mu^\star_A: \pi_1(\cC_A) \rightarrow B_d^\star$ is surjective. Moreover, the composition of the map $\pi_1(\cC_A\cap \{c_0$=$c_2$=$1\})\rightarrow \pi_1(\cC_A)$ with $\mu_A$ is surjective.
\end{Corollary}

\begin{proof}
This follows from Proposition \ref{prop:correspondence} and Lemmas  \ref{lem:generatorArtin},  \ref{lem:generatorsbraidgroup} and \ref{lem:trinomial}.
\end{proof}

\begin{Remark}
Using the tropical construction above, we could prove Theorems \ref{thm:reduced} and \ref{thm:contractible} for trinomials without having to work out any relations in the braid group $B_d^\star$. Unfortunately, these tropical techniques cannot be applied to arbitrary support sets. For supports $A$ with four elements, the set $L$ is to be replaced with the phase-tropical plane $H\subset \R^3$. For particular supports $A$, the geodesic $\Arg\big(\phi_c(\C^\star)\big)$ intersects $\Arg(H)$ in less than $d$ components for any choice of $c$. Therefore, there is no generalisation of the loop $\gamma$ constructed above.
\end{Remark}

\section{Proofs of Theorem \ref{thm:reduced} and \ref{thm:contractible}}\label{sec:proofs}

In this section, the support $A\subset \N$ is reduced and with extremal elements $0$ and $d$. In order to prove the main theorems, we specialise to the case of all trinomials $\{0,p,d\}\subset A$ and implement an analogue of the Euclidean algorithm in  $B_d^\star$ in the name of Proposition \ref{prop:euclide}.

For any divisor $k\geqslant 2$ of $d$ and any $j \in \{1,\cdots,d\}$, define $J_{k,j}:=\big\lbrace \ell \in \{1,\cdots,d\} \big|$ $ \ell \equiv j \mod k \big\rbrace$ and the braid 
$$[b_{k,j}]:= \prod_{\ell\in J_{k,j}} [b_\ell].$$
Observe that the above elements $[b_\ell]$ commute so that $[b_{k,j}]$ is well defined.

\begin{Proposition}\label{prop:euclide}
Let $k, \ell \geqslant 2$ be two distinct divisors of $d$ and denote $q:=\lcm(k,\ell)$. 
The subgroup of $B_d^\star$ generated by $\cup_{1\leqslant j\leqslant d} \left\lbrace [b_{k,j}], [b_{\ell,j}] \right\rbrace$ contains $\cup_{1\leqslant j\leqslant d} \left\lbrace [b_{q,j}]\right\rbrace$.
\end{Proposition}

\begin{proof}[Proof of Theorems \ref{thm:reduced} and \ref{thm:contractible}]
Choose $p \in A\setminus\{0, d\}$. Then, the space of conditions $\cC_A$ contains the space of conditions $\cC_{\tilde A}$ for $\tilde A := \{0,p,d\}$. In particular, the image of $\mu^\star_A$ contains the image of $\mu^\star_{\tilde A}$. If we denote $a:=\gcd(p,d)$, $p':=\frac{p}{a}$ and $d':=\frac{d}{a}$, then $\im\big(\mu^\star_{\tilde A}\big)$ contains the pullback by the covering $x\mapsto x^a$ of every braid in $\im\big(\mu^\star_{\{0,p',d'\}}\big)$ since every polynomial in $\cC_{\tilde A}$ is the composition of a polynomial in $\cC_{\{0,p',d'\}}$ with the latter covering. In particular, the image of $\mu^\star_{\tilde A}$ contains the pullback $[b_{a,j}]\in B_d^\star$
of $[b_j]\in B_{d'}^\star$ for all possible $j$, according to Corollary \ref{cor:trinomial}. It follows now from Proposition \ref{prop:euclide} that $\im(\mu^\star_A)$ contains $[b_{q,j}]$
where $q:=\lcm\big(\big\lbrace \frac{d}{\gcd(p,d)} \, \big| \, p\in A\setminus\{0, d\}\big\rbrace
\big)$. Since $A$ is reduced, the integer $q$ is equal to $d$ and $[b_{q,j}]=[b_j]\in B_d^\star$. 
Observe moreover that we can choose the loops in $\pi_1(\cC_{\{0,p',d'\}})$ giving rise to the $[b_j]
$-s in $B_{d'}^\star$ to be in $\{c_0$=$c_2$=$1\}\subset\C^{\{0,p',d'\}}$ by Corollary 
\ref{cor:trinomial}, and then the loops giving rise to $[b_{d',j}]$ and $[b_{q,j}]=[b_j]\in B_d^\star$ 
are in $\{c_0$=$c_2$=$1\}\subset\C^{\tilde A}$ and $\{c_0$=$c_2$=$1\}\subset \C^{A}$ 
respectively. This fact together with Lemma \ref{lem:generatorArtin} imply Theorem 
\ref{thm:contractible}. Theorem \ref{thm:reduced} follows from Lemma 
\ref{lem:generatorsbraidgroup} and the fact that $[\tau]$ is always contained in $\im(\mu^\star_A)
$. Indeed, the element $[\tau]$ is the image under $\mu^\star_A$ of the loop $\theta \mapsto p_
\theta(x)=x^d-e^{i\theta}$, $\theta \in [0,2\pi]$.
\end{proof}

We will now proceed to the proof of Proposition \ref{prop:euclide}. In order to do so, we will need the following terminology.

\begin{Definition}\label{def:simplesparse}
A braid $[b]\in B_d^\star$ is \emph{simple} (respectively \emph{sparse}) if it can be written as a composition $[b]=[b_{j_1}] \circ [b_{j_2}]\circ \cdots \circ  [b_{j_k}]$ for a sequence of integers $1\leqslant j_1 < j_2<\cdots<j_k \leqslant d$ such that the difference between consecutive integers is at least $2$ (respectively $3$). In particular, all the elements in the above decomposition of $[b]$ commute. Any simple braid $[b]$ is determined by the collection $\{j_1,\cdots,j_k\}$ that we refer to as the \emph{support} of $[b]$. Denote by \emph{$[b_J]$} the simple braid with support $J$ and define the support $\emph{J_n}=\left\lbrace k \in 2\Z+1 \; \big| \;   1\leqslant k \leqslant 2n-1\right\rbrace$ of cardinality $n$, for any integer $n\leqslant \lfloor d/2 \rfloor$.
\end{Definition}

Observe that the braid $[b_{k,j}]$ is always simple and that  it is sparse provided that $k\geqslant 3$. For the sake of brevity, let us introduce the notation $\emph{g\star g'}:= g \circ g' \circ g^{-1}$.

\begin{Lemma}\label{lem:swap}
Let $[b]$ and $[\beta]$ be two sparse braids in $B_d^\star$ with respective supports $J$ and $J'$. Then, the braid
\[ [b]\bullet [\beta] : = [b]^{-1}\circ [\beta]^{-1}\circ [b]\circ [\beta]\circ [b] \]
is the simple braid with support $J \bullet J'$ consisting of the union of   $J\cap J'$ with the elements in $J$ at distance at least $2$ from $J'$ with the elements of $J'$ at distance exactly $1$ from $J$.
\end{Lemma}
%\a{[Denoting the set of points at the distance at most (respectively, exactly) one from $J$ by $B_1(J)$ (respectively $S_1(J)$), the support of $J \bullet J'$ equals $B_1(J)$ minus $S_1(J')$. Could this language simplify the study of cases a little, if later we decide to write it in more detail for the journal version?)]}\l{we could indeed use this terminology in future.}

\begin{proof}
First, we observe that the set $J \bullet J'$ is the support of a simple braid, that is, two consecutive indices are at distance at least $2$ from each other. This is a straightforward consequence of the fact that $[b]$ and $[\beta]$ are sparse.

If $j\in \{1,\cdots,d\}$ is such that $\{j-1,j\}$ is disjoint from $J\cup J'$,  then the $j^{th}$ strand of $[b]\bullet [\beta]$ is straight. Therefore, there are at most $15$ configurations to check: each of the indices $j$ and $j-1$ are in one of the sets $(J\cup J')^c$, $J\setminus J'$, $J'\setminus J$ or $J\cap J'$ and we should discard the case when both indices belong to $(J\cup J')^c$ by the previous argument. Among these configurations, some are prevented by the fact that $[b]$ and $[\beta]$ sparse and some configurations are redundant. All the relevant configurations are depicted in Figure \ref{fig:proof}.
\end{proof}

\begin{figure}[h]
\centering
\scalebox{0.75}{
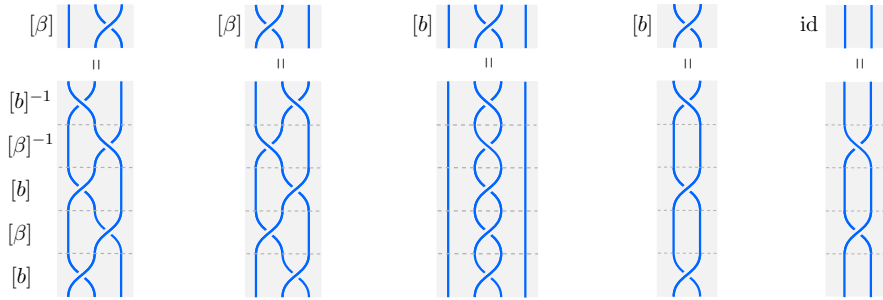
}
\caption{Local pictures for the braid $[b]\bullet [\beta]$. The first two configurations on the left correspond to indices in $J'$ at distance $1$ from $J$, the third to indices in $J\cap J'$, the fourth to indices in $J$ at distance at least $2$ from $J'$ and the last one to indices in $J'$ at distance at least $2$ from $J$.}
\label{fig:proof}
\end{figure}

\begin{Lemma}\label{lem:compactsupport}
Let $G\subset B_d^{\star}$ be a subgroup that is invariant under conjugation by $[\tau]$. Assume further that $G$ contains either $[b_j]\circ[b_j+1]^{-1}$ or $[b_j]\circ[b_j+2]^{-1}$ for some $j$. If $G$ contains a simple braid $[b_J]$, then $G$ contains the simple braid $[b_{J'}]$ for any support $J'$ such that $\# J=\# J'$.
\end{Lemma}

\begin{proof}
Denote $n:=\#J$. We will show that $G$ contains $[b_J]$ if and only if it contains the simple braid $[b_{J_n}]$, see Definition \ref{def:simplesparse}.
%We will show that if $n:=\#J$, then there is $g,g'\in G$ such that $g'\circ[b_{J}]\circ g$ is the simple braid $[b_{J'}]$ with $J'=\{1,3,\cdots,2n+1\}$.

Assume first that $G$ contains $[b_j]\circ[b_{j+1}]^{-1}$ for some $j$. Since $G$ is invariant under $[\tau]$, then it contains $[b_j]\circ[b_{j+1}]^{-1}$ for any $j$. In particular, if $j\in J$ and $j-2 \notin J$ (recall that $j-1 \notin J$ since $[b_J]$ is simple), then $[b_{J}]\circ [b_{j-1}]\circ[b_j]^{-1}=[b_{J'}]$ where $J':=(J\setminus \{j\})\cup \{j-1\}$. Repeating this procedure as many time as necessary, we can shift the smallest index in $J$ to $1$, then the second smallest index to $3$ and so on, until we obtain the sought simple braid $[b_{J_n}]$. The result follows.

Assume now that $G$ contains $[b_j]\circ[b_{j+2}]^{-1}$ for some $j$, and thus for any $j$. Since $G$ is invariant under $[\tau]$, there is no loss of generality in assuming that $1\in J$. As in the previous paragraph, we want to  shift the indices of $J$ one by one to the left using elements in $G$, until we obtain the support $J_n$. For any $j\in J\setminus \{1\}$, there are two cases to consider: either $\{j-3,j-2\} \cap J = \varnothing$ or $j-3\in J$ and $j-2\notin J$ (the remaining case $j-2 \in J$ corresponds to the situation where we cannot shift $j$ any further to the left). If we are in the first case, then $[b_{J}]\circ [b_{j-2}]\circ[b_j]^{-1}=[b_{J'}]$ where $J':=(J\setminus \{j\})\cup \{j-2\}$. In the second case, we have that 
%\[ \big(([b_{j-3}]^{-1}\circ[b_{j-1}])\circ[b_J]^{-1}\circ([b_{j-3}]\circ[b_{j-1}]^{-1})\big)\circ[b_J]\circ\big(([b_{j-3}]^{-1}\circ[b_{j-1}])\circ[b_J]\circ([b_{j-3}]\circ[b_{j-1}]^{-1})\big) \]
\[ \sigma_0:=\big(([b_{j-3}]^{-1}\circ[b_{j-1}])\star [b_J]^{-1}\big)\star[b_J]\]
is equal to $[b_{J'}]$ with $J'=(J\setminus \{j\})\cup \{j-1\}$, see Figure \ref{fig:proof4}. In both cases, the element $[b_{J'}]$ is in $G$ if and only if $[b_J]$ is. This proves inductively that $[b_{J_n}]\in G$ if and only if $[b_J]\in G$. The result follows.
%If $j\in J$ and $\{j-3,j-2\} \cap J = \varnothing$, then $[b_{J}]\circ [b_{j-2}]\circ[b_j]^{-1}=[b_{J'}]$ where $J':=(J\setminus \{j\})\cup \{j-2\}$. Therefore, we can shift the smallest index in $J\setminus \{1\}$ to either $3$ or $4$ by multiplying $[b_J]$ by elements in $G$. If we obtain $3$, we can start shifting the next smallest index in $J$. If we obtain $4$, then we 
\end{proof}

\begin{figure}[h]
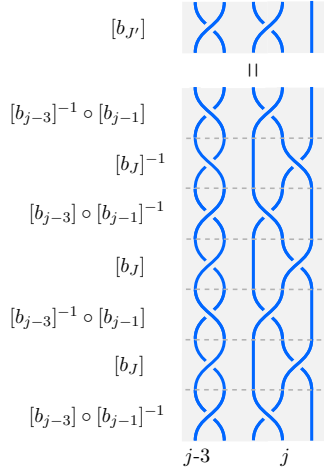
\caption{The braid $[\sigma_0]$. We only picture the strands from $j-3$ to $j+1$ since $[b_{j-3}]\circ [b_{j-1}]^{-1}$ restricts to the identity on the remaining strands. The whole products restricts therefore to $[b_J]^{-1}\circ[b_J]^2=[b_J]$ on the latter strands.}
\label{fig:proof4}
\end{figure}

\begin{proof}[Proof of Proposition \ref{prop:euclide}]
Denote by $G$ the group generated by $\cup_{1\leqslant j\leqslant d}\left\lbrace [b_{k,j}], [b_{\ell,j}]\right\rbrace$. The group $G$ is invariant by conjugation by $[\tau]$, since its generating set is. Therefore, it suffices to show that $G$ contains $[b_{q,j}]$ for a single $j$.
% since all the $[b_{q,j}]$-s are equivalent to each other under conjugation by $[\tau]$.

Denote $a:=\gcd(k,\ell)$ and define $k'$, $\ell'$, and $b$  such that $k=ak'$, $\ell=a\ell'$ and $d=abk'\ell'$. First, observe that if $k$ divides $\ell$ or $\ell$ divides $k$, there is nothing to prove. Thus, we can assume that $k'>\ell' \geqslant 2$ are coprime. Second, observe that if the statement is true for the triple $k$, $\ell$ and $d':=ak'\ell'$, then the statement is also true for $k$, $\ell$ and $d$. This is easily seen by using the covering $x\mapsto x^b$ from $\C^\star$ to itself. Therefore, there is not loss of generality in assuming that $d=ak'\ell'$. In that case $q=d$ and $[b_{q,j}]= [b_j]$. Thus, we have to show that $[b_j]$ is an element in $G$ for some $j$.

Assume first that $a \geqslant 3$. Under the present assumptions, the braid $[b_{k,1}]$ is the 
product of the $\ell'$ elements $[b_i]$ for $i\equiv 1 \mod k$ and $[b_{\ell,1}]$ is the product of 
the $k'$ elements $[b_i]$ for $i\equiv 1 \mod \ell$. Observe that $J_{k,1}\cap J_{\ell,1}=\{1\}$ and 
that the distance $\vert j'-j'' \vert $ between elements $j'\in J_{k,1}$  and $j''\in J_{\ell,1}$  is 
always divisible by $a\geqslant 3$. By Lemma \ref{lem:swap}, we obtain that $[b]:=[b_{k,1}]\bullet 
[b_{\ell,2}]\in G$ has support $(J_{k,1} \setminus \{1\})\cup \{2\}$. In turn, we deduce that $[b_{k,
1}]\circ [b]=[b_1]\circ[b_2]^{-1}$ is an element in $G$. Thus, we are in position to apply Lemma 
\ref{lem:compactsupport} and obtain that both $[b_{J_{\ell'}}]$ and $[b_{J_{k'}}]$ are in $G$. The 
result follows now from the Euclidean algorithm. Indeed, we obtain that $[b_{J_{k'-\ell'}}]\in G$ as 
the conjugation of $[b_{J_{k'}}]\circ[b_{J_{\ell'}}]^{-1}\in G$ by the appropriate power of $[\tau]$.
Inductively, we obtain that $[b_{J_{\delta}}]\in G$ where $\delta:=\gcd(\ell',k')$. Since the latter $gcd$ is $1$ by assumption, we conclude that $[b_1]\in G$. The result follows.

Assume now that $a = 2$. We will show that $G$ contains $[b_j]\circ [b_{j+2}]^{-1}$ for some $j$, apply Lemma \ref{lem:compactsupport} and conclude with the Euclidean algorithm as above. Assume first that $\ell'\geqslant 3$. Since $k'$ and $\ell'$ are coprime, the $\ell'$ elements of $J_{k,1}$ achieve the $\ell'$ classes modulo $\ell$ that are divisible by $a$. In particular, there are exactly two elements $\jm,\jp\in J_{k,1}$ such that $\jm+2\in J_{\ell,1}$ and $\jp-2\in J_{\ell,1}$. All elements in $J_{k,1}\setminus \{1,\jm,\jp\}$ are at distance at least $4$ from $J_{\ell,1}$. Since $a=3$ and $\ell'\geqslant 3$, two consecutive element in $J_{\ell,j}$ are at distance at least $6$ 
from each other. Therefore, the index $\jm$ is at distance at least $2$ from $J_{\ell,2}$ while $1$ 
and $\jp$ are at distance $1$. The remaining indices of $J_{k,1}$ are at distance at least $3$ from 
$J_{\ell,2}$. It follows that $[b_{k,1}]\bullet[b_{\ell,2}]\in G$ is equal to $[b_{J'}]$ with 
$J':= J_{k,1}\setminus \{1,\jp\})\cup \{2,\jp-1\}$. In turn, we have that $([\tau]\star[b_{J'}])\bullet 
[b_{k,1}]\in G$ is equal to  $[b_{J''}]$ with $J'':= (J_{k,1}\setminus \{1\})\cup \{3,\}$. We conclude 
that $G$ contains $[b_{k,1}]\circ [b_{J''}]^{-1}=[b_1]\circ [b_3]^{-1}$ and the result follows. If now $
\ell'=2$ (implying that $\ell=4$), then $J_{k,1}=\{1,k+1\}$ and $\{1, k-1, k+3\}\subset J_{4,1}$ since $k$ is divisible by $2$ but not
by $4$. It follows that $[\tau]\star([b_{k,1}]\bullet[b_{4,2}])\in G$ is equal to $[b_3]\circ[b_{k+1}]
$. Dividing $[b_{k,1}]$ by the latter element gives $[b_1]\circ [b_3]^{-1}\in G$. The result follows.

Assume now  that $a=1$. We postpone the case $\ell=2$ and assume for now that $\ell\geqslant 3$. Since $k$ and $\ell$ are coprime, 
%the $\ell$ elements of $J_{k,j}$ have distinct reduction modulo $\ell$. In particular, 
there are exactly two elements $\jm$ and $\jp$ in $J_{k,1}$ that are at distance $1$ from $J_{\ell,1}$, namely the two elements with respective reduction $0$ and $2$ modulo $\ell$. Since $J_{k,1}\cap J_{\ell,1}=\{1\}$, any element in $J_{k,1}\setminus \{1,\jm,\jp\}$ is at distance at least $2$ from $J_{\ell,1}$. 
%Since $k\geqslant 4$ and $\ell\geqslant 3$, Lemma \ref{lem:swap} applies and the braid 
We deduce that $[b_{k,1}]\bullet [b_{\ell,1}]\in G$ is equal to $[b_{J^{(4)}}]$ with $J^{(4)}:=(J_{k,1}\setminus \{\jm,\jp\}) \cup \{\jm+1,\jp-1\}$ and then that $([\tau]\star[b_{J^{(4)}}])\bullet [b_{k,1}]\in G$ is the simple braid with support $(J_{k,1}\setminus \{\jm\})\cup\{\jm+2\}$. Dividing $[b_{k,1}]$ by the latter element, we obtain that $[b_{\jm}]\circ [b_{\jm+2}]^{-1}\in G$. The result follows again from Lemma \ref{lem:compactsupport} and the Euclidean algorithm.

It remains to treat the case $\ell=2$ and $k$ odd. For $k=3$, define the following elements of $G$
\[ [\sigma_1] := [\tau]^{-1}\star \big( [b_{2,2}]^{-1}\circ [b_{3,1}]\circ [b_{2,1}]^{-1}\big) \circ [b_{3,2}]\circ [b_{3,1}]\]
and  $ [\sigma_2] := [b_{2,1}]^{-1}\circ \big([b_{3,1}]\bullet [b_{2,1}]\big)$, see Figure \ref{fig:proof2}. Then, we have that $ [\sigma_2]^{-1}\circ [\sigma_1] = [b_{3}]\in G$ and the result follows. 

\begin{figure}[H]
\centering
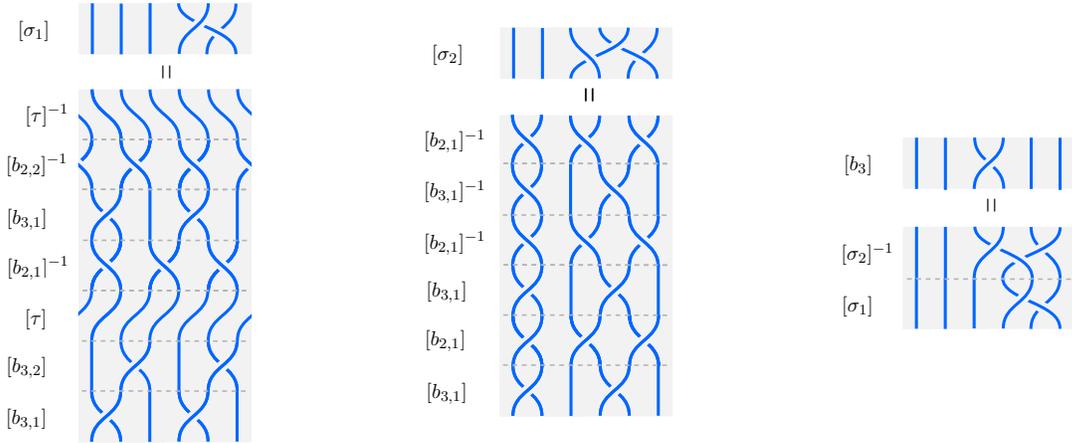
\caption{The braids $[\sigma_1]$ and $[\sigma_2]$.}
\label{fig:proof2}
\end{figure}

For $k\geqslant 5$, define the following elements of $G$\\
\[
\begin{array}{l}
\left[\sigma_3\right] :=  \big([b_{k,2}]^{-1} \circ[b_{k,1}]^{-1}\circ [b_{2,1}]^{-1}\big)\star [b_{k,1}], \quad
\left[\sigma_4\right] := [b_{k,1}] \circ [b_{k,2}]\circ [b_{k,3}] \circ [b_{k,4}]\\ \\
\text{and } \left[\sigma_5\right] := \big([\sigma_4]\circ([\tau]^{k+1}\star [\sigma_3])\circ [b_{k,1}]\big)^{-1}\star\big([\tau]\star [\sigma_3]\big).
\end{array}
\]
The latter elements are depicted in Figure \ref{fig:proof3}. In particular, we obtain that $\sigma_5=[b_2]\circ[b_{k+4}]$ and thus that $[b_1]\circ[b_{k+3}]\in G$. In turn, we obtain $[b_{k,1}]\circ ([b_1]\circ[b_{k+3}])^{-1}=[b_{k+1}]\circ[b_{k+3}]^{-1}\in G$. The result follows from Lemma \ref{lem:compactsupport} and the Euclidean algorithm.
%
%Now, if $k>5$, the conjugation $([b_{k,k+5}]\circ[b_{k,k+4}])^{-1}\circ[\sigma_5] \circ([b_{k,k+5}]\circ[b_{k,k+4}])\in G$ is equal to $[b_1]\circ[b_{k+5}]$. Repeating this procedure if necessary, we obtain that $[b_1]\circ[b_{2k-1}]$ is an element of $G$, for any $k\geqslant 5$. Conjugation by $[\tau]^2$, we deduce that $[b_1]\circ[b_{3}]\in G$. Eventually, the element $\sigma_6\in G$ defined as 
%\[ [b_{k,2}]^{-1}\circ[b_{k,3}]\circ([b_1]\circ[b_3])^{-1}\circ[b_{k,2}]\circ ([b_1]\circ[b_3])\circ[b_{k,3}]^{-1}\circ [b_{k,2}] \l{false}\]
%is equal to $[b_1]$. This concludes the proof.
\end{proof}

\begin{figure}[H]
\centering
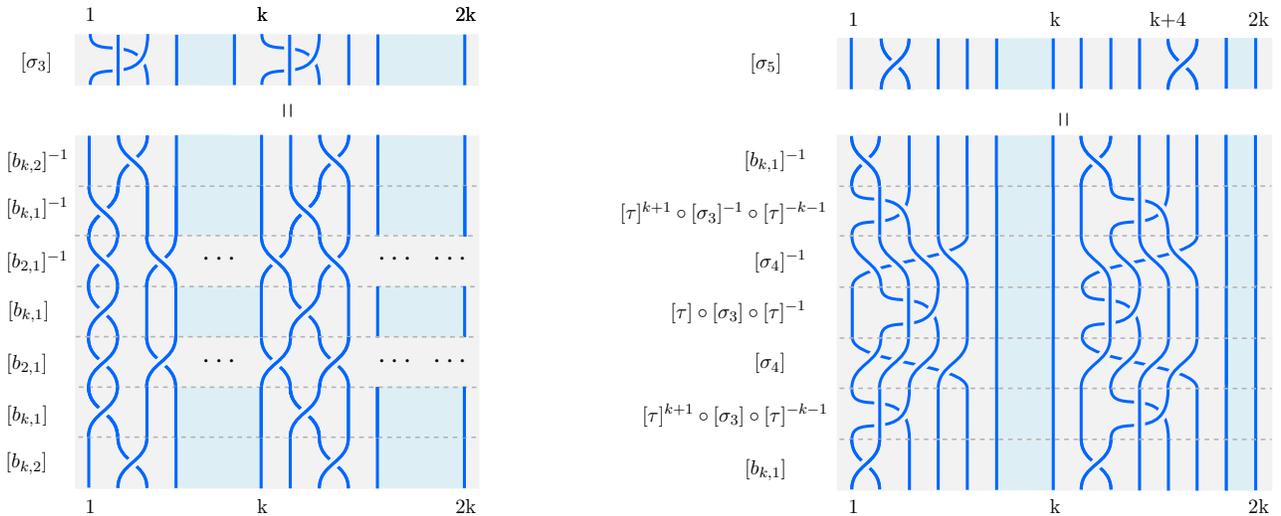
\caption{The braids $[\sigma_3]$, $[\sigma_4]$ and $[\sigma_5]$ (some of the blue areas are not present for $k=5$).}
\label{fig:proof3}
\end{figure}

\bibliographystyle{alpha}
\bibliography{Draft}

\begin{thebibliography}{{Coh}80}

\bibitem[{Arn}71]{Arn71}
V.~I. {Arnol'd}.
\newblock {On some topological invariants of algebraic functions.}
\newblock {Trans. Mosc. Math. Soc. 21 (1970), 30-52 (1971).}, 1971.

\bibitem[Art47]{Ar}
E.~Artin.
\newblock Theory of braids.
\newblock {\em Ann. of Math. (2)}, 48:101--126, 1947.

\bibitem[{Bes}01]{Be}
D.~{Bessis}.
\newblock {Zariski theorems and diagrams for braid groups.}
\newblock {\em {Invent. Math.}}, 145(3):487--507, 2001.

\bibitem[Che75]{Ch}
D.~Cheniot.
\newblock Un th\'{e}or\`eme du type de {L}efschetz.
\newblock {\em Ann. Inst. Fourier (Grenoble)}, 25(1):xi, 195--213, 1975.

\bibitem[{Coh}80]{Co}
S.~D. {Cohen}.
\newblock {The Galois group of a polynomial with two indeterminate
  coefficients.}
\newblock {\em {Pac. J. Math.}}, 90:63--76, 1980.

\bibitem[DL81]{DL}
I.~{Dolgachev} and A.~{Libgober}.
\newblock {On the fundamental group of the complement to a discriminant
  variety.}
\newblock {Algebraic geometry, Proc. Conf., Chicago Circle 1980, Lect. Notes
  Math. 862, 1-25 (1981).}, 1981.

\bibitem[EL18]{AL}
A.~{Esterov} and L.~{Lang}.
\newblock {Inductive irreducibility of solution spaces and systems of equations
  whose Galois group is a wreath product}.
\newblock {\em arXiv e-prints}, page arXiv:1812.07912, Dec 2018.

\bibitem[Est19]{E18}
A.~Esterov.
\newblock Galois theory for general systems of polynomial equations.
\newblock {\em Compos. Math.}, 155(2):229--245, 2019.

\bibitem[FM11]{farbmarg}
B.~{Farb} and D.~{Margalit}.
\newblock {\em {A primer on mapping class groups.}}
\newblock Princeton, NJ: Princeton University Press, 2011.

\bibitem[FN62]{FN}
E.~Fadell and L.~Neuwirth.
\newblock Configuration spaces.
\newblock {\em Math. Scand.}, 10:111--118, 1962.

\bibitem[GKZ08]{GKZ}
I.M. Gelfand, M.M. Kapranov, and A.V. Zelevinsky.
\newblock {\em {Discriminants, resultants, and multidimensional determinants.
  Reprint of the 1994 edition.}}
\newblock {Modern Birkh\"auser Classics. Boston, MA: Birkh\"auser. x, 523~p.},
  2008.

\bibitem[{Lan}19]{L19}
L.~{Lang}.
\newblock {Monodromy of rational curves on toric surfaces}.
\newblock {\em arXiv e-prints}, page arXiv:1902.08099, Feb 2019.

\bibitem[Lib90]{Lib90}
A.~Libgober.
\newblock On topological complexity of solving polynomial equations of special
  type.
\newblock In {\em Transactions of the {S}eventh {A}rmy {C}onference on
  {A}pplied {M}athematics and {C}omputing ({W}est {P}oint, {NY}, 1989)},
  volume~90 of {\em ARO Rep.}, pages 475--478. U.S. Army Res. Office, Research
  Triangle Park, NC, 1990.

\end{thebibliography}

\bigskip
\bigskip
\noindent
A. Esterov\\
National Research University Higher School of Economics\\
Faculty of Mathematics NRU HSE, Usacheva str., 6, Moscow, 119048, Russia\\
\textit{Email}: aesterov@hse.ru\\

\noindent
L. Lang\\
Department of Mathematics, Stockholm University, SE - 106 91 Stockholm, Sweden.\\
\textit{Email}: lang@math.su.se

\end{document}